\documentclass[12pt]{article}
\usepackage{amsmath,amssymb,amsthm,tikz}
\usepackage{amsfonts,graphicx,latexsym,enumerate}
\usepackage[notref,notcite]{} 
\usepackage{changes}
 
 \usepackage{color,graphicx,url}        

\numberwithin{equation}{section}

  \usepackage[top=2.54cm, bottom=2.54cm, left=2.54 cm, right=2.54 cm]{geometry}

\newcommand{\eqa}{\begin{eqnarray}}
\newcommand{\ena}{\end{eqnarray}}
\newcommand{\eq}{\begin{equation}}
\newcommand{\en}{\end{equation}}
\newcommand{\eqs}{\begin{eqnarray*}}
\newcommand{\ens}{\end{eqnarray*}}

\def\r{\varrho} 

\newcommand{\R}     {\mathbb{R}} 
 
\newcommand{\N}     {\mathbb{N}}

\newcommand{\prob}[1]{\bbP\left({#1}\right)}
\newcommand{\expected}[1]{\bE\left[{#1}\right]}

\def\1{{\mathchoice {1\mskip-4mu\mathrm l}      
{1\mskip-4mu\mathrm l} 
{1\mskip-4.5mu\mathrm l} {1\mskip-5mu\mathrm l}}} 
 
\def\comment#1{} 
 
\newtheorem{theorem}{Theorem}[section] 
\newtheorem{lemma}[theorem]{Lemma}

\newtheorem*{definition}{Definition}

\newcommand{\heap}[2]{\genfrac{}{}{0pt}{}{#1}{#2}} 
 
\newcommand{\s}{\sigma}
\renewcommand{\d}{{\rm d}} 
 
\newcommand{\eps}{\varepsilon}

\newcommand{\Bcal}  {{\mathcal B}}
\newcommand{\Ccal}   {{\mathcal C }}

\newcommand{\Gcal}   {{\mathcal G }}

\newcommand{\Jcal}   {{\mathcal J }} 

\newcommand{\Tcal}   {{\mathcal T }} 
 
\def\ignore#1{}

\def\Def{=}

\def\bbP{\mathbb{P}}

\def\bE{\mathbf{E}}

\def\parent#1{#1^{-}}

\def\Ma{{\rm Mass}}
\def\MMa{{\rm AM}}

\def\We{{\rm Weight }}

\def\T{{{\rm Tree}_{r}}}
\def\R{{{\rm RAN}_t}}

\renewcommand{\emph}{\textit}

\begin{document}

\thispagestyle{empty}
\def\thefootnote{\fnsymbol{footnote}}
\title{
Longest paths in random Apollonian networks 
and largest $r$-ary subtrees of random  $d$-ary recursive trees}
\author{
Andrea Collevecchio\thanks{Supported by ARC Discovery Project grant DP140100559.}\\
{\small School of Mathematical Sciences, Monash University, and}\\
{\small Ca' Foscari University, Venice}\\
{\small {\tt  andrea.collevecchio@monash.edu }} 
\and
Abbas  Mehrabian\\
{\small Department of Combinatorics and Optimization, University of Waterloo}\\
{\small {\tt amehrabi@uwaterloo.ca}}
\and
Nick Wormald\thanks{Supported by Australian Laureate Fellowships grant FL120100125.}\\
{\small School of Mathematical Sciences, Monash University} \\
{\small {\tt  nick.wormald@monash.edu }} }
\date{}

\maketitle

\vspace{-0.8cm} 
\noindent {\bf Abstract.} 
Let $r$ and $d$ be positive integers with $r<d$.
Consider a random $d$-ary tree constructed as follows.
Start with a single vertex, and in each time-step choose a uniformly random leaf and give it $d$ newly created offspring.
Let $\Tcal_t$ be the tree produced after $t$ steps.
We show that there exists a fixed $\delta<1$ depending on $d$ and $r$ such that almost surely for all large $t$, every $r$-ary subtree of $\Tcal_t$ has less than $t^{\delta}$ vertices.

The proof involves analysis that also yields a related result.
Consider the following iterative construction of a random planar triangulation.  
Start with a triangle embedded in the plane.
In each step, choose a bounded face uniformly at random,
add a vertex inside that face and join it to the vertices of the face. In this way, one face is destroyed and  three new faces are created.
After $t$ steps, we obtain a random triangulated plane graph with $t+3$ vertices,
which is called a {random Apollonian network}.
We prove that there exists a fixed $\delta<1$, such that 
{eventually}
every path in this graph {has} length less than $t^{\delta}$, which verifies a conjecture of Cooper and Frieze.

\bigskip


\bigskip 

\newpage
\setcounter{footnote}{0}
\def\thefootnote{\arabic{footnote}}

\section{Introduction}\label{intro}
In this paper we study two  important random graph models.
The first one is a so-called \emph{random $d$-ary recursive tree}, defined as follows.
Let $d>1$ be a positive integer.
Consider a random $d$-ary tree evolving as follows. 
At time 0 it consists of exactly one vertex, $\r$. 
In the first step $\r$ gives birth to $d$ offspring.
In each subsequent step we pick, uniformly at random, a vertex with no offspring and connect it with exactly $d$  offspring.
At time $t$ this random tree is denoted by $\Tcal_{t}$. 
See {Drmota}~\cite{random_trees} for more on random $d$-ary recursive trees.
Let $r$ be a fixed positive integer smaller than $d$ and let 
$S_t$ denote the size of the largest (possibly non-unique) $r$-ary subtree of $\Tcal_t$. 

We say  that a sequence of events $\{A_k,\,  k \in \N\}$  occurs \emph{eventually } {(for    large $k$)} if there exists an almost surely (a.s.)  finite random variable $N$ such that   $A_k $ occurs for all  $k \ge N$. 
In this paper all logarithms are natural.
\begin{theorem}
\label{thm:maximalsubtree}
There exists a fixed $\delta<1$ such that {$S_t < t^{\delta}$ eventually}. 
\end{theorem}


In Section~\ref{sec:explicitrary} we show we can take
$$\delta = 1 - \frac{d-r}{ed^{2d}\log \left(11 d \log d \right)}$$
in this theorem.

The second object we study is a popular random graph model for generating planar graphs with power law properties, which is defined as follows.
Start with a triangle embedded in the plane.
At each step, choose a bounded face uniformly at random,
add a vertex inside that face and join it to the vertices on the face. 
{In this way, one face is destroyed and  three new faces are created.}
We call this operation \emph{subdividing} the face.
After $t$ steps, we have a (random) triangulated plane graph $\R$ with $t+3$ vertices, $3t+3$ edges,   $2t+1$ bounded faces, and 1 unbounded face.
The random graph $\R$ is called a \emph{random Apollonian network}. 

Random Apollonian networks were defined by Zhou, Yan, and Wang~\cite{define_RANs} 
(see Zhang, Comellas, Fertin, and Rong~\cite{high_RANs} for a generalization to higher dimensions),
where it was proved that the diameter of $\R$ is probabilistically bounded above by a constant times $ \log t$.
It was shown in~\cite{define_RANs,RANs_powerlaw} that $\R$ exhibits a power law degree distribution for large $t$.
The average distance between two vertices in a typical $\R$ was shown to be $\Theta(\log t)$ by Albenque and Marckert~\cite{RANs_average_distance},
and a central limit theorem was proved by
Kolossv\'{a}ry, Komj\'{a}ty, and V\'{a}g\'{o}~\cite{istvan}.
The degree distribution, $k$ largest degrees, $k$ largest eigenvalues (for fixed $k$), and   diameter
were studied by Frieze and Tsourakakis~\cite{first}. 
The asymptotic value of the diameter of a typical $\R$ was determined in~\cite{we_rans}.
We continue this line of research by
studying the asymptotic properties of the longest (simple) paths in $\R$.

Let $\mathcal{L}_t$ be a random variable denoting the number of vertices in a longest path in $\R$.
All the limits in this paragraph are as $t\to\infty$.
Frieze and Tsourakakis~\cite{first} conjectured there exists a fixed $\delta>0$ such that
$\prob{\delta t \le \mathcal{L}_t < t}\to 1 $.
Ebrahimzadeh, Farczadi, Gao, Mehrabian, Sato, Wormald, and Zung~\cite{we_rans} refuted this conjecture and showed there exists a fixed $\delta>0$ such that $\prob{\mathcal{L}_t < t / (\log t)^{\delta}}\to1$.
Cooper and Frieze~\cite{longpaths_new} improved this result by showing that for every constant $c<2/3$, we have $\prob{\mathcal{L}_t \le t \exp(-\log^c t)}\to 1$,
and conjectured there exists a fixed $\delta<1$ such that 
$\prob{\mathcal{L}_t \le t^{\delta}}\to1$.
The second main result of this paper is the following, which in particular confirms this conjecture.

\begin{theorem}
\label{thm:longest_upper}
There exists a fixed $\delta<1$ such that 
{$\mathcal{L}_t < t^{\delta}$ eventually}.
\end{theorem}
\noindent We can take $\delta = 1-4\times 10^{-8}$,
as shown in Section~\ref{sec:explicit_rans}.
 
Regarding lower bounds, it was proved in~\cite{we_rans} that $\mathcal{L}_t > (2t+1)^{\log 2 / \log 3}$ deterministically,
and that $\expected{\mathcal{L}_t} = \Omega\left(t^{0.88}\right)$.

We prove the two main theorems by studying a third object, an infinite tree {with weighted vertices}, which is introduced in Section~\ref{sec:preliminaries}.
Then we prove Theorems~\ref{thm:maximalsubtree} and~\ref{thm:longest_upper} in 
Sections~\ref{sec:r-ary} 
and~\ref{sec:rans}, respectively.
{Note that both of these theorems are existential. 
In Section~\ref{sec:appendix} we give explicit bounds for the values of $\delta$ in these theorems.}

\section{Subtrees of an infinite $d$-ary tree}
\label{sec:preliminaries}

Fix positive integers $r,d$ with $r < d$.
Let $ \Tcal$ be an infinite rooted $d$-ary tree. 
Denote the root by $\r$. 
We denote by $[\nu, \mu]$ the vertices in the  path connecting $\nu$ to $\mu$, including these two vertices. 
For a vertex $\nu$, 
denote its distance from $\r$ by $|\nu|$,
and its offspring by $\nu i$, with $i \in \{1, 2, \ldots, d\}$. 
For $\nu\neq\r$, denote by $\parent{\nu}$ the parent of $\nu$, i.e. the neighbour $\mu$ of $\nu$ with $|\mu| = |\nu| -1$.


To each vertex $\nu$ assign a random variable $X_{\nu} \in (0,1]$, satisfying the following properties. 
We have $X_\r = 1$.  
The random variables $X_{\nu}$, for $\nu \in V(\Tcal)\setminus\{\r\}$, are identically distributed. 
Moreover we assume that  the vectors  $(X_{\nu1}, X_{\nu2}, \ldots,  X_{\nu d})$ are identically distributed and independent,
and that 
$\sum_{i=1}^{d} X_{\nu i}=1$.
For any vertex $\nu$, define the random variable
\begin{equation}
\label{def_upsilonfirst}
\Upsilon_{\nu} \Def \min \{ X_{\nu i_1}+ X_{\nu i_2} +\ldots + X_{\nu i_{d-r}}\colon 1\le i_1< i_2 <\ldots< i_{d-r} \le d\}\:,
\end{equation}
and let $\Upsilon \Def \Upsilon_{\nu}$ for an arbitrary $\nu$.

For each vertex $\nu \in V(\Tcal)$, define 
$$ \Ma(\nu) \Def \prod_{\s \in [\r, \nu]} X_\s,$$
and for any set of vertices $A \subset V(\Tcal)$, 
{let}
$\Ma(A) = \sum_{\nu \in A} \Ma (\nu)$.

Given non-negative integer  $n$, consider \emph{level $n$} of $ \Tcal$, i.e. the set of vertices at distance $n$ from $\r$. Denote by 
$\Gcal_{n, r}$ the collection of subsets of {at most} $r^{n}$ vertices at level $n$, with the additional property that they belong to the same $r$-ary subtree.  

\pagebreak

The main result of this section is the following.
\begin{lemma}\label{ABb} 
Let $\lambda, \kappa$ be positive constants satisfying
\begin{equation}
\label{kappalambdacond}
d \: \kappa^{\lambda  } \: \expected{\left(1-\Upsilon\right)^{\lambda}} < 1 \:.
\end{equation}
Then eventually for large $n$,
$$ \max_{C \in \Gcal_{n, r}}  \Ma(C) \le \kappa^{-n}.$$
\end{lemma}

For 
{a given positive integer $n$ and a positive number $\kappa$}, define the event 
$$ \Ccal_{n, \kappa}  \Def \bigcap_{\nu \colon |\nu|=n} \left\{\prod_{\s \in [\r, \parent{\nu}]} \big(1 - \Upsilon_\s\big)^{-1} \ge \kappa^n \right\},
$$
and define the random variable
\begin{equation}\label{enn}
 N_1 = N_1(\kappa) = \min \{ n \colon \Ccal_{j, \kappa} \mbox{ holds for all $j \ge n$}\}.
\end{equation}
We set $N_1 =  \infty$ if the set {on} the right-hand side is empty.

\begin{lemma}\label{lem-e-h}
If $\lambda,\kappa> 0$ satisfy \eqref{kappalambdacond}, then $ N_1(\kappa)$ is a.s.\ finite.
\end{lemma}

\begin{proof}
By the first Borel-Cantelli lemma, it is enough to show that 
\begin{equation}
\label{borelnts}
\sum_{n=1}^\infty d^n \bbP\left(\prod_{\s \in [\r, \parent{\nu}]} \big(1 - \Upsilon_\s\big)^{-1} < \kappa^n\right)<\infty \:,
\end{equation}
where $\nu=\nu(n)$ denotes an arbitrary vertex with $|\nu|=n$.
Since  the $\Upsilon_{\s}$ are independent and $\lambda > 0$,  the above probability  is by Markov's inequality
\begin{align*}
\bbP\left(\prod_{\s \in [\r, \parent{\nu}]} \big(1 - \Upsilon_\s\big)^{\lambda} > \kappa^{-\lambda n}\right)
& 
\le \expected{\prod_{\s \in [\r, \parent{\nu}]} \big(1 - \Upsilon_\s\big)^{\lambda}} \kappa^{\lambda n} \\
& = \left(\expected{\big(1 - \Upsilon\big)^{\lambda}} \kappa^{\lambda} \right)^n \:.
\end{align*}
The inequality \eqref{borelnts} now follows from \eqref{kappalambdacond}.
\end{proof}



For each vertex $\nu$, we define its  {\it adjusted mass}, denoted by $\MMa(\nu)$,
as follows. For the root, $\MMa(\r)=1$, and for all other vertices $\nu$,
$$ 
\begin{aligned}
\MMa(\nu) &\Def \prod_{\s \in [\r, \parent{\nu}] }  X_{\s } \left(\frac 1{1-\Upsilon_\s}\right). 
\end{aligned}
$$
For any $A \subset V(\Tcal)$, 
{let}
$\MMa(A) = \sum_{\nu \in A} \MMa (\nu)$.
\begin{lemma}\label{adjm} 
For every positive integer $n$ and every
$C \in \Gcal_{n,r }$ we have 
$ \MMa(C) \le 1$.
\end{lemma}
\begin{proof}
Let  $C\in \Gcal_{n, r}$. Define $\T(C)$ to be  the $r$-ary subtree of $\Tcal$ whose  leaves are the vertices of $C$. For each vertex $\nu$ of $\T(C)$,   denote  its set of offspring in $\T(C)$ by $\nu_{\rm off}$. Then by the definition of $\Upsilon$ in~\eqref{def_upsilonfirst}, 
\[
\Ma(\nu_{\rm off} )  \le (1- \Upsilon_{\nu}) \Ma(\nu).
\]
 Thus $  \MMa(\nu_{\rm off} ) \le \MMa(\nu).$ 
Hence, for any $1\le k \le n$, we have  
$$ \sum_{\heap{\mu \in \T(C)}{|\mu| = k}} \MMa (\mu) \le \sum_{\heap{\nu \in \T(C)}{|\nu| = k-1}} \MMa (\nu).$$
Iterating this,  we get 
\begin{align*}\MMa(C) &= \sum_{\nu \in C} \MMa (\nu) \le \MMa(\r) =1.\qedhere
\end{align*}
\end{proof}

\begin{proof}[Proof of Lemma~\ref{ABb}.]
Recall the definition of $N_1$ from \eqref{enn}. 
Lemma~\ref{lem-e-h} implies that $N_1$ is a.s. finite.
If  $n \ge N_1$, then  for any $ C \in \Gcal_{n,r}$ we have 
$$ \Ma(C) \le \kappa^{-n} \MMa(C) \le \kappa^{-n},$$
where the last inequality is a consequence of Lemma~\ref{adjm}.
\end{proof}


\section{Largest $r$-ary subtrees of random $d$-ary   trees}
\label{sec:r-ary}
\newcommand{\node}[1]{{\mathbf{v}}^{#1}}
As the Beta and Dirichlet distributions play an important role in this paper, we recall their definitions.

\begin{definition}[Beta and Dirichlet distributions]
Let $\Gamma(t) \Def \int_0^\infty x^{t-1} {\rm e}^{-x} \d x$.
For positive numbers $\alpha,\beta$, 
a random variable is said to be distributed as Beta$(\alpha, \beta)$ if it has density 
$$ \frac{\Gamma(\alpha + \beta)}{\Gamma (\alpha) \Gamma(\beta)} \: x^{\alpha -1} \: (1-x)^{\beta -1} \qquad \mathrm{\ for\ }x \in (0,1)\:.
$$
The multivariate generalization of the Beta distribution is called the Dirichlet distribution.
Let $\alpha_1,\alpha_2,\dots,\alpha_n$ be positive numbers. The Dirichlet($\alpha_1, \alpha_2, \ldots \alpha_n$) distribution has support on the set 
$$ \Big\{(x_1, x_2, \ldots, x_n) \colon x_i \ge 0 \mathrm{\ for\ }1\le i\le n, \mathrm{\ and\ } \sum_{i=1}^n  x_i =1\Big\} \:,$$ 
and its density at point $(x_1,x_2,\dots,x_n)$ equals
$$
\frac{\Gamma\Big(\sum_{i=1}^n \alpha_i\Big) }{\prod_{i=1}^n \Gamma(\alpha_i)} \: \prod_{j=1}^n x_j^{\alpha_j -1}.
$$
Note that if the vector $(X_1, X_2, \ldots, X_n)$ is distributed as Dirichlet($\alpha_1, \alpha_2, \ldots, \alpha_n)$, then the marginal distribution of $X_i$ is Beta($\alpha_i, \sum_{ j \neq i} \alpha_j$).
\end{definition}

Let $r$ and $d$ be fixed positive integers with $r<d$.
Let $(B_1, B_2, \ldots, B_d) $ be a  random vector distributed as Dirichlet$(1/(d-1), 1/(d-1), \ldots, 1/(d-1))$, and define the random variable $\Upsilon$ as
\begin{equation}
\label{def_upsilon}
\Upsilon \Def  \min \{ B_{i_1}+ B_{i_2} \ldots + B_{i_{d-r}}\colon 1\le i_1< i_2 <\ldots< i_{d-r} \le d\}\:.
\end{equation}

%

The main theorem of this section is the following.

\begin{theorem}\label{mainmain}
Let  $r$ and $d$ be  fixed positive integers with $r<d$,
and let $S_t$ denote the size of the largest $r$-ary subtree of 
a random $d$-ary recursive tree at time $t$. 
Let $\tau,\kappa,\lambda$ be positive constants satisfying
\begin{align}
e\:d\:\log\tau & < (d-1) \tau^{1/(d-1)} \:,\label{m_cond}\\
d \, \kappa^{\lambda  } \: \expected{\left(1-  \Upsilon \right)^{\lambda}} & < 1 \:, \label{kappa_cond}
\end{align}
and let $n = \lfloor \log t / \log \tau \rfloor $.
There exists a constant $K$ such that  eventually (for large $t$)
$$S_t \le K \left( r^{n} +  t \kappa^{-n} \right)\:.$$
\end{theorem}

Before proving this theorem, we show it implies Theorem~\ref{thm:maximalsubtree}.

\begin{proof}[Proof of Theorem~\ref{thm:maximalsubtree}.]
We show there exist positive constants
$\tau,\kappa,\lambda$ satisfying
\eqref{m_cond}, \eqref{kappa_cond}, and $\kappa>1$;
then we would have $\tau > e^{d-1} > r$, and
the theorem follows from Theorem~\ref{mainmain} by choosing any
$\delta \in \big(\max\{1 - \log \kappa / \log \tau,\log r / \log \tau\}, 1\big)$.
As \eqref{m_cond} holds for all large enough $\tau$,
it suffices to show there exist $\kappa>1$ and $\lambda>0$ satisfying
\eqref{kappa_cond}.
Since $\lim_{\eps\to 0 }\prob{\Upsilon<\eps}=0$,
we have
$\expected{\left(1-  \Upsilon \right)^{\lambda}}\to 0$
as $\lambda\to\infty$.
In particular, there exists $\lambda>0$ such that 
$\expected{\left(1-  \Upsilon \right)^{\lambda}} <  1/d$.
Then, we can let 
$$\kappa
= \left(d \expected{\left(1-  \Upsilon \right)^{\lambda}}\right)^{-1/(2\lambda)}\:,
$$
which is strictly larger than 1.
\end{proof}
 
In the rest of this section, $\Tcal$ denotes an infinite $d$-tree rooted at $\r$. We view the random recursive $d$-ary tree $\Tcal_t$ as a subtree of $\Tcal$.
At each time step, we assign a \emph{weight} to each vertex.  For each   $t$ and   each  vertex $\nu$ of $\Tcal_t$, define  
$ \aleph(\nu, t)$ to be the number of  vertices $\mu$ of $\Tcal_t$ such that $\nu \in [\r, \mu]$. This is the number of vertices in the ``branch'' of $\Tcal_t$ containing $\nu$, including $\nu$.
 Set
$$\We(\nu, t) \Def \frac{\aleph(\nu, t) -1}d $$
if $\nu \in V(\Tcal_t)$,
and $\We(\nu, t) \Def 0$ if $\nu \notin V(\Tcal_t)$. 
Note that this is the number of non-leaf vertices in this branch at time $t$. 
\begin{lemma}\label{Nk1}
There exist random variables $\{B_{\nu}\}_{{\nu}\in V(\Tcal)}$, with $B_\r=1$, such that for any positive integer $t$ and any ${\nu}\in V(\Tcal)$,
$\We({\nu}, t)$ is stochastically dominated by a  Binomial$\left(t, \prod_{\s \in [\r, {\nu}]} B_{\s}\right)$ random variable. 
Moreover, the vectors $(B_{\nu 1}, B_{\nu 2}, \ldots, B_{\nu d})$ are independent for $\nu \in V(\Tcal)$, and are distributed as Dirichlet$(1/(d-1), 1/(d-1), \ldots, 1/(d-1))$.
\end{lemma}
\begin{proof} 
{Consider a vertex $\nu\ne \r$ and a positive integer $t$ such that $\nu \in V(\Tcal_t)$.}
Note that {at time $t$,} the number of leaves in the branch at $v$ is $(d-1)\We(\nu, t)+ 1$. Hence, given that at time $t+1$   the weight of $\parent{\nu}$ increases, the probability, conditional on the past,  that the weight of $\nu$ increases at the same time, is equal to
$$ \frac{(d-1)\We(\nu, t)+ 1}{(d-1)\We(\parent{\nu}, t)+1}.$$
Each time a weight increases, its increment is exactly 1.  
By identifying $\nu$ with one colour and its siblings with another colour, 
{the evolution of the numerator of the above expression over time} can be modelled using an Eggenberger-P\'olya urn, with initial conditions $(1, d-1)$  and reinforcement equal to $d-1$.
Moreover, the urns corresponding to distinct vertices are mutually independent.


The limiting distribution describing the Eggenberger-P\'olya urn is 
well known, but we require bounds applying for all $t$. 
To this end, we can express the number of times a given colour is chosen by time $t$  in an Eggenberger-P\'olya urn as a mixture of binomials with respect to a Beta distribution. 
See, for example, {Pemantle}~\cite[Lemma~1]{PE88}.
In this case, given the initial conditions $(1, d-1)$ and reinforcement $d-1$, the mixture is with respect to Beta($1/(d-1), 1$). 
{This means that}
to each vertex  $\nu$ we can assign a random variable $B_\nu$ distributed as  Beta($1/(d-1),  1$), 
{such that $\We(\nu, t)$ conditional on $B_{\nu}$}
is binomially distributed with parameters $\We(\parent{\nu}, t)-1$ and $B_{\nu}$. 
Set $B_\r=1$ and note that $\We(\r, t)=t$. By induction, 
$\We(\nu, t)$, conditional on 
{$\{B_{\s}\}_{\s \in [\r, \nu]}$}, is stochastically smaller than a Binomial$\left(t, \prod_{\s \in [\r, \nu]} B_{\s}\right)$. 

By the Eggenberger-P\'olya urn representation we also infer that the joint distribution of $(B_{\nu 1}, B_{\nu 2}, \ldots ,B_{\nu d})$ is Dirichlet$(1/(d-1), 1/(d-1), \ldots, 1/(d-1))$ for all $\nu$. 
See, for example,~\cite[Lemma~1]{PE88}.
\end{proof}


\begin{lemma}
\label{productbetaproperties}
Let $B_1,\dots,B_n$ be independent Beta$(1/(d-1),1)$ random variables.
For all positive $\beta$ we have
$$\mathbb{P}\left(\prod_{i=1}^{n} B_{i} \le \beta^{n}\right) \le 
\left(
\frac{e\log(1/\beta)\beta^{1/(d-1)}}{d-1}
\right)^n.
$$
\end{lemma}

\begin{proof}
If $\beta \ge e^{1-d}$ then
the right-hand side is at least 1, so we may assume that
$0<\beta<e^{1-d}$.
We use Chernoff's technique.
Let $\lambda \in(-1/(d-1),0)$ be a parameter which will be specified later.
We have
$$ \mathbf{E}[B_1^{\lambda}] =  \frac {\Gamma(d/(d-1))}{\Gamma(1/(d-1)) \Gamma(1)} \int_{0}^{1} x^{\lambda} x^{- 1+1/(d-1)}  \d x=  \frac{1}{(d-1) \lambda +1}\:.$$
Hence by Markov's inequality 
{and since the $B_i$ are independent,}
\begin{align*} 
\mathbb{P}\left(\prod_{i=1}^{n} B_{i} \le \beta^{n}\right)
 = \mathbb{P}\left(\prod_{i=1}^{n} B_{i}^{\lambda} \ge \beta^{\lambda n}\right) 
\le 
\prod_{i=1}^{n} \frac{\expected{B_1^{\lambda}}}{\beta^{\lambda}}
=
\left(
\frac{1}{\beta^{\lambda}((d-1) \lambda +1)}
\right)^n.
\end{align*}
To 
{minimize the right-hand side} we choose
$\lambda =  -1/(d-1)-1/\log \beta $,
which is in the correct range since 
$0<\beta<e^{1-d}$.
This gives
\begin{align*}
\mathbb{P}\left(\prod_{i=1}^{n} B_{i} \le \beta^{n}\right) &\le 
\left(
\frac{e\log(1/\beta)\beta^{1/(d-1)}}{d-1}
\right)^n\:.\qedhere
\end{align*}
\end{proof}

We now prove Theorem~\ref{mainmain}.

\begin{proof}[Proof of Theorem~\ref{mainmain}.]
Let $\{B_\nu\}_{\nu\in V(\Tcal)}$ be as given by Lemma~\ref{Nk1}.
Denote by
 $\Gcal_{n, r}$ the collection of subsets of the vertices of $\Tcal$ at level $n$, with the property that they belong to the same $r$-ary  subtree. 
We apply Lemma~\ref{ABb} with $\Ma$ defined using $X_\sigma=B_\sigma$.   
{Since \eqref{kappa_cond} holds}, 
we {conclude that}
eventually for large $n$,
\begin{equation}\label{ABb1}
 \max_{C \in \Gcal_{n, r}}  \Ma(C) \le \kappa^{-n}.
\end{equation}


By Lemma~\ref{Nk1}, $\We(\nu, t)$ is stochastically dominated by a Binomial$(t, \Ma(\nu))$. 
Chernoff's bound for binomials implies (see, e.g.,~\cite[Theorem~2.3(b)]{concentration})
$$
\prob{\We(\nu,t) \ge 2 t \Ma(\nu) \:|\: \Ma(\nu) \ge q}
\le \exp(-tq) \:,
$$
for  every positive $q$.

Since $\tau$ satisfies \eqref{m_cond}, there exists $\tau_1<\tau$ satisfying
\begin{equation} 
{e d \log \tau_1} < (d-1)\tau_1^{1/(d-1)} \:. \label{m1_cond}
\end{equation}
Let $\beta = 1 / \tau_1$.
By Lemma~\ref{productbetaproperties}, for any vertex $\nu$ at level $n$  
$$
\bbP(\Ma(\nu)  <  \beta ^{n}) \le \left(
\frac{e\log(1/\beta)\beta^{1/(d-1)}}{d-1}
\right)^n \:.
$$
Note that \eqref{m1_cond} implies that the term in brackets is a constant smaller than $1/d$.

We have 
\begin{equation*}
\begin{aligned}
 \bbP\Big(\bigcup_{\mu \colon |\mu|=n}   \big\{\We(\mu, t) \ge   2 t \Ma(\mu) \big\}\Big) 
 &\le  d^{n} \bbP\big( \We(\nu , t) \ge 2 t \Ma(\nu) )\\
 &\le  d^n\bbP(\Ma(\nu)  <  \beta ^{n}) + 
 d^{n} \exp\left(-t\beta^n\right). \\
\end{aligned}
\end{equation*}
The last expression is summable in $n$, 
as $t^{1/n}\beta \ge \tau\beta$, and $\tau\beta$ is a constant larger than~1.
By the first Borel-Cantelli lemma and \eqref{ABb1}, there exists 
a constant $K$ such that eventually for large $t$ we have
$S_t  \le K\left( r^{n} + t \kappa^{-n} \right)$.
\end{proof}

\section{Longest paths in random Apollonian networks}
\label{sec:rans}
We define a tree $\Tcal_t$, called  the \emph{$\triangle$-tree} of $\R$, as follows.
There is a one to one correspondence between the triangles in $\R$ and the vertices of $\Tcal_t$.
For every triangle  $\triangle$ in $\R$,
we denote its corresponding vertex in $\Tcal_t$ by $\node{\triangle}$.
To build $\Tcal_t$, start with a single root vertex $\r$, which   corresponds to the initial triangle of $RAN_t$.
Wherever a triangle $\triangle$ is subdivided into triangles $\triangle_1$, $\triangle_2$, and $\triangle_3$,
generate three {offspring} $\node{\triangle_1}$, $\node{\triangle_2}$, and $\node{\triangle_3}$ for $\node{\triangle}$,
and extend the  correspondence in the natural manner.
Note that $\Tcal_t$  is a random $3$-ary recursive tree 
as defined in 
{the introduction} 
and $\Tcal_t$ has $3t+1$ vertices and $2t+1$ leaves.
The leaves of $\Tcal_t$ correspond to the bounded faces of $\R$.
Let $\Tcal$ denote an infinite $3$-ary tree rooted at $\r$.
We view $\Tcal_{t}$ as a subtree of $\Tcal$.
For each vertex $\nu \in V(\Tcal)$, the \emph{grand-offspring} of $\nu$ are its descendants at level $|\nu|+2$. 
For a triangle $\triangle$ in $\R$, $I(\triangle)$ denotes the set of vertices of $\R$ that are \emph{strictly inside} $\triangle$.

The following lemma was proved in \cite[Lemma~3.1]{we_rans},
{and introduces the connection with the setup in Section~\ref{sec:r-ary}.}
\begin{lemma}
Let $\Tcal_{t}$ be  the  $\triangle$-tree  of $\R$.
Let $\node{\triangle}$ be a vertex of $\Tcal_t$ with nine grand-offspring $\node{\triangle_1},\node{\triangle_2},\dots,\node{\triangle_9}$ in $V(\Tcal_t)$.
Then the vertex set of a path in $\R$ intersects at most eight of the $I(\triangle_i)$'s.
\label{lem:9children}
\end{lemma}

We say that a finite subtree $\Jcal$ of $\Tcal$ is \emph{buono} if each vertex of $\Jcal$   has at most eight grand-offspring in $\Jcal$.  

Assume the four vectors $(A_1,A_2,A_3)$, $(B_1,B_2,B_3)$, $(B_4,B_5,B_6)$ and $(B_7,B_8,B_9)$ are i.i.d\ random vectors distributed as Dirichlet$(1/2, 1/2, 1/2)$. Define the random variable 
\begin{equation}
\label{def_upsilonA}
\Upsilon = \min \{ A_i B_j : 1\le i \le 3, 1 \le j \le 9\}\:.
\end{equation}
The main theorem of this section is the following.

\begin{theorem}\label{mainmainA}
Let $\tau,\kappa,\lambda$ be positive constants satisfying
\begin{align}
3e \log\tau & < 2\sqrt {\tau} \:,\label{m_cond_rans}\\
9 \:  \kappa^{\lambda  } \: \expected{\left(1-  \Upsilon \right)^{\lambda}} &< 1 \label{kappa_cond_rans}\:,
\end{align}
and let $n$ be the largest even integer smaller than $\log t / \log \tau$.
Then, there exists a constant $ K$, such that  eventually for large $t$, the largest  buono subtree of $\Tcal_t$ has at most
$K\left( 8^{n/2} +  t \kappa^{-n/2} \right)$
vertices.
\end{theorem}

We show how this theorem implies Theorem~\ref{thm:longest_upper}.

\begin{proof}[Proof of Theorem~\ref{thm:longest_upper}.]
Following the proof of Theorem~\ref{thm:maximalsubtree},  
we can find positive constants $\tau, \kappa, \lambda$ satisfying the conditions of Theorem~\ref{mainmainA}, with $\kappa>1$ and  $\tau > 8$.
Choose any $\delta \in 
\left(\max\{1 - \log(\kappa)/(2\log \tau),\log (8) / 2 \log \tau\},1\right)$.


Let $P$ be a path in $\R$ and let $R(P)$ denote the set of vertices $\node{\triangle}$ of $\Tcal_{t}$ such that $I(\triangle)$ contains some vertex of $P$.
By Lemma~\ref{lem:9children},
$R(P)$ induces a buono subtree of $\Tcal_{t}$.
Hence, using Theorem~\ref{mainmainA} for the second inequality, eventually for large $t$ we have
$$
|V(P)|
\le
3 + |R(P)|
\le 3 + K\left(8^{n/2} +  t \kappa^{-n/2} \right)
< t^{\delta}\:,
$$
as required.
\end{proof}
 

{The rest of this section is devoted to the proof of Theorem~\ref{mainmainA}.}
For each time $t$ and   vertex $\nu$ of $\Tcal_t$, let  
$ \aleph(\nu, t)$ denote the number of  vertices $\mu$ of $\Tcal_t$ such that $\nu \in [\r, \mu]$.
Let
$$ \We(\nu, t) \Def \frac{\aleph(\nu, t) -1}3 $$
if $\nu\in V(\Tcal_t)$, and $\We(\nu, t) =0$ if $\nu \notin V(\Tcal_t)$. 

By Lemma~\ref{Nk1},
there exist random variables $\{B_\nu\}_{\nu\in V(\Tcal)}$, 
such that for any positive integer $t$ and any $\nu\in V(\Tcal)$,
 $\We(\nu, t)$ is stochastically dominated by a  Binomial$\left(t, \prod_{\s \in [\r, \nu]} B_{\s}\right)$. 
Moreover, $B_\r=1$ and for all $\nu \in V(\Tcal)$, 
the joint vectors $(B_{\nu 1}, B_{\nu 2}, B_{\nu 3})$ are independent and distributed as Dirichlet $(1/2, 1/2, 1/2)$.
Define $ \Ma(\nu) \Def \prod_{\s \in [\r, \nu]} B_{\s}$. 

Denote by
 $\Bcal_{n}$ the collection of subsets of $V(\Tcal)$ at level $n$, with the property that they belong to the same buono  subtree.  Note that each element of $\Bcal_{2n}$ has at most $8^n$ vertices. 
\begin{lemma}\label{ABb1A} 
Let $\lambda, \kappa$ be positive constants satisfying \eqref{kappa_cond_rans}.
Eventually, for large $n$,
$$ \max_{C \in \Bcal_{2n}}  \Ma(C) \le \kappa^{-n}.$$
\end{lemma}
\begin{proof}
Let  $\Tcal'$ be an  infinite rooted 9-ary tree obtained from $\Tcal$ as follows. The vertices of $\Tcal'$ are the vertices of $\Tcal$ at even levels. A vertex $\mu$ is an offspring of $\nu$ in $\Tcal'$ if $\mu$ is a grand-offspring of $\nu$ in $\Tcal$.  To each vertex $\mu$ assign the random variable $X_\mu = B_\mu B_{\parent{\mu}}$. Buono subtrees of $\Tcal$ are translated into 8-ary  subtrees of $\Tcal'$. 
Using Lemma~\ref{ABb} concludes the proof.
\end{proof}

We now prove Theorem~\ref{mainmainA}.
The proof is similar to that of Theorem~\ref{mainmain}.

\begin{proof}[Proof of Theorem~\ref{mainmainA}.]
Recall that $\We(\nu, t)$ is stochastically dominated by a Binomial (t,$\Ma(\nu)$). 
Chernoff bound for binomials implies (see, e.g.,~\cite[Theorem~2.3(b)]{concentration})
$$
\prob{\We(\nu,t) \ge 2 t \Ma(\nu) \:|\: \Ma(\nu) \ge q}
\le \exp(-tq)
$$
for  every positive $q$.

Since $\tau$ satisfies \eqref{m_cond_rans}, there exists $\tau_1<\tau$ satisfying
\begin{equation}
3e \log \tau_1  < 2\sqrt {\tau_1} \:. \label{m1_cond_rans}
\end{equation}
Let $\beta = 1 / \tau_1$.
By Lemma~\ref{productbetaproperties}, for any vertex $\nu$ at level $n$ we have
$$
\bbP(\Ma(\nu)  <  \beta ^{n}) \le 
\left(
\frac{e\log(1/\beta)\sqrt{\beta}}{2}
\right)^n \:.
$$
Note that \eqref{m1_cond_rans} implies that the term in brackets is a constant smaller than $1/3$.

We have 
\begin{equation*}
\begin{aligned}
 \bbP\Big(\bigcup_{\mu \colon |\mu|=n}   \big\{\We(\mu, t) \ge   2 t \Ma(\mu) \big\}\Big) 
 &\le  3^{n} \bbP\big( \We(\nu , t) \ge 2 t \Ma(\nu) )\\
 &\le  3^n\bbP(\Ma(\nu)  <  \beta ^{n}) + 
 3^{n} \exp\left(-t\beta^n\right). \\
\end{aligned}
\end{equation*}
The last expression is summable in $n$, 
as $t^{1/n}\beta \ge\tau\beta$, and $\tau\beta$ is a constant larger than 1.
By the first Borel-Cantelli lemma and Lemma~\ref{ABb1A}, 
there exists a constant $K$ such that eventually for large $t$, 
the largest  buono subtree of $\Tcal_t$ has at most
$K\left( 8^{n/2} +  t \kappa^{-n/2} \right)$
vertices.
\end{proof}

\section{Appendix: Explicit bounds}
\label{sec:appendix}

\subsection{Explicit bound for Theorem~\ref{thm:maximalsubtree}}
\label{sec:explicitrary}

In this section we prove an explicit version of Theorem~\ref{thm:maximalsubtree}.
For the proof we will need 
the following inequalities (see, e.g., Laforgia~\cite[Equations (2.2) and (2.3)]{LAF84}),
valid for all $p,q>0$ and $0<\sigma<1<\iota<2$:
\begin{equation}
\left(p + \iota/2\right)^{\iota-1} <
\frac{\Gamma(p+\iota)}{\Gamma(p+1)}, \mathrm{and\ \quad}
\frac{\Gamma(q+\sigma)}{\Gamma(q+1)} <
\left(q+\sigma/2\right)^{\sigma - 1} \:.
\label{laforgia}
\end{equation}

\begin{theorem}
Let  $r$ and $d$ be  fixed positive integers with $r<d$,
and let $S_t$ denote the size of the largest $r$-ary subtree of 
a random $d$-ary recursive tree at time $t$. 
Let $$\delta = 1 - \frac{d-r}{ed^{2d}\log \left(11 d \log d \right)} \:.$$
Then   eventually  $S_t < t^{\delta}$.
\end{theorem}

\begin{proof}
In view of the proof of Theorem~\ref{thm:maximalsubtree},
it suffices to show there exist positive constants $\tau,\kappa,\lambda$ satisfying the following inequalities:
\begin{align}
e \: d \: \log \tau & < (d-1) \tau^{1/(d-1)}\:, \label{m_cond_explicit}\\
d \: \kappa^{\lambda  } \: \expected{\left(1-  \Upsilon \right)^{\lambda}} & < 1 \:, \label{kappa_cond_explicit} \\
\log r / \log \tau < 1 - \log \kappa / \log \tau & < \delta \:, \label{delta_cond_explicit}
\end{align}
where $\Upsilon$ is defined in \eqref{def_upsilon}.

Set 
$$\lambda = e d^{2d-2}/(d-r),\quad \mathrm{and}\quad\kappa = \exp\left(\frac{1}{\lambda(d-1)}\right) \:.$$
Let $B$ be a Beta($1/(d-1),1$) random variable.

We have
$$\prob{\Upsilon \le \eps}
\le d \prob{B \le \frac{\eps}{d-r}}
= d \left(\frac{\eps}{d-r}\right)^{1/(d-1)} \:.$$
Therefore,
\begin{align*}
\expected{\left(1- \Upsilon \right)^{\lambda}} 
& = \int_{0}^{1}
\prob{\left(1- \Upsilon \right)^{\lambda} \ge x} d x \\
& = \int_{0}^{1}
\prob{\Upsilon \le 1 - x^{1/\lambda}} d x \\
& \le 
\int_{0}^{1}
d \left(\frac{1 - x^{1/\lambda}}{d-r}\right)^{1/(d-1)} dx\:.
\end{align*}
Using the change of variables $y = x^{1/\lambda}$, we find
$$
\int_{0}^{1}
\left({1 - x^{1/\lambda}}\right)^{1/(d-1)} d x 
= \frac{\Gamma(d/(d-1))\Gamma(\lambda+1)}{\Gamma(\lambda+d/(d-1))} 
<  \lambda^{-1/(d-1)} \:,
$$
where we have used the inequalities in~\eqref{laforgia}
with $p=\lambda,\iota=d/(d-1),q=1,$ and $\sigma=1/(d-1)$.
Thus,
\begin{align*}
d \: \kappa^{\lambda} \: \expected{\left(1-  \Upsilon(r) \right)^{\lambda}} 
& < 
\frac{d^2 \kappa^{\lambda} }{(\lambda(d-r))^{1/(d-1)}} 
=1 \:,
\end{align*}
so \eqref{kappa_cond_explicit} holds.

Let $\tau_0=(11 d \log d)^{d-1}$ and notice that
$ 
e \: d \: \log \tau_0  < (d-1) \tau_0^{1/(d-1)}
$  
 since $d\ge2$.
Moreover,
$$
\frac{\log\kappa}{\log \tau_0} = 
\frac{d-r}{(d-1)^2 e \log \left(11 d \log d \right)d^{2(d-1)}}
>
\frac{d-r}{ed^{2d}\log \left(11 d \log d \right)} \:.
$$
Choose $\tau$ larger than $\tau_0$ such that
$$
\frac{\log\kappa}{\log \tau} >
\frac{d-r}{ed^{2d}\log \left(11 d \log d \right)} \:.
$$
We also have $\log r + \log \kappa < \log \tau$.
Thus \eqref{m_cond_explicit} and \eqref{delta_cond_explicit} hold as well, and the proof is complete.
\end{proof}

\subsection{Explicit bound for Theorem~\ref{thm:longest_upper}}
\label{sec:explicit_rans}
In this section we provide an explicit value for $\delta$ in Theorem~\ref{thm:longest_upper}.
For $\lambda>1$, define
$$g(\lambda)=\frac{9\lambda}{2(\lambda-1)^{3/2}}
\Big(\sqrt{\pi} + \sqrt{\pi} \log \left(\lambda-1\right)/2
+ 4/9\Big) \:.$$

\begin{lemma}
\label{productbetapropertiesA}
Let $\Upsilon$ be defined as \eqref{def_upsilonA}, and let $\lambda>1$.
Then $\expected{\left(1- \Upsilon \right)^{\lambda}}
< g(\lambda)$.
\end{lemma}

\begin{proof}
Let $B_1$ and $B_2$ be independent Beta$(1/2,1)$ random variables.
The density function of each of $B_1$ and $B_2$ is $1/(2\sqrt{x})$ if $x\in(0,1)$ and 0 elsewhere, hence we have
\begin{align*}
\prob{B_1B_2\le \eps}
& = \int_0^{1} 
\left(\int_0^{\min\{1,\eps/x\}} \frac{1}{2\sqrt{y}}\mathrm{d}y\right)
\frac{1}{2\sqrt{x}}\mathrm{d}x 
=\sqrt{\eps} (1 + \log(1/\eps)/2) \:.
\end{align*}

Thus 
\begin{align*}
\expected{\left(1- \Upsilon \right)^{\lambda}} 
& = \int_{0}^{1}
\prob{\left(1- \Upsilon \right)^{\lambda} \ge x} d x \\
& = \int_{0}^{1}
\prob{\Upsilon \le 1 - x^{1/\lambda}} d x \\
& \le 9 
\int_{0}^{1}
\prob{B_1 B_2 \le 1 - x^{1/\lambda}} d x \\
& = \frac{9}{2} 
\int_{0}^{1} \sqrt{1 - x^{1/\lambda}} \: \log \left(\frac{e^2} { 1 - x^{1/\lambda}}\right) dx\,. 
\end{align*}
With $y=(\lambda-1)(1-x^{1/\lambda})$,
we find  
\begin{align*}
\expected{\left(1- \Upsilon \right)^{\lambda}}  
& 
\le \frac{9\lambda}{2(\lambda-1)^{3/2}}  \int_{0}^{\lambda-1} \sqrt{y} \: \log \left(\frac{e^2(\lambda-1)} {y}\right)
\left(1-\frac{y}{\lambda-1}\right)^{\lambda-1}
dy \\
& 
< \frac{9\lambda}{2(\lambda-1)^{3/2}}  \int_{0}^{\lambda-1} \sqrt{y} \: \log \left(\frac{e^2(\lambda-1)} {y}\right) e^{-y}
dy \:.
\end{align*}
We have 
$$
\int_{0}^{\infty} \sqrt{y} \: \log \left({e^2(\lambda-1)}\right) e^{-y} dy
= 
\log \left({e^2(\lambda-1)}\right) \sqrt{\pi}/2 \:,
$$
and
$$
\int_{0}^{\lambda-1} 
\sqrt{y} \: \log \left(1/y\right) e^{-y} dy
\le
\int_{0}^{1}
\sqrt{y} \: \log \left(1/y\right) e^{-y} dy
\le
\int_{0}^{1}
\sqrt{y} \: \log \left(1/y\right) dy
=4/9 \:,\\
$$
concluding the proof.
\end{proof}

Set
$\lambda = 10^6$,  
$\kappa = (9g(\lambda))^{-1/\lambda}$,
$\tau = 720$, and $\delta  = 
1-4\times 10^{-8}$.
 It is easily verified that
${3e} \log \tau  <2\sqrt{\tau}$ 
 and 
$\delta>\max\{1 - \log(\kappa)/2\log \tau,\log (8) / 2 \log \tau\}$.
Moreover, Lemma~\ref{productbetapropertiesA}
implies that $9 \: \kappa^{\lambda} \: \expected{\left(1- \Upsilon \right)^{\lambda}} < 1$.
As in the proof of Theorem~\ref{thm:longest_upper},
we get that  $\mathcal{L}_t < t^{\delta}$ eventually.

\bibliographystyle{plain}
\bibliography{webgraph}
 
\end{document}